\documentclass[11pt,reqno]{amsart}
\usepackage[colorlinks=true,
pdfstartview=FitV, linkcolor=cyan, citecolor=magenta,
urlcolor=blue]{hyperref}
\usepackage{amsmath,amsfonts,latexsym,amssymb}
\usepackage{mathrsfs}
\usepackage[latin1]{inputenc}
\usepackage[T1]{fontenc}
\usepackage{ae,aecompl}
\usepackage{braket}
\usepackage{comment}
\usepackage{color}
\usepackage{graphicx}
\usepackage{subfig}
\usepackage{bbm}

\usepackage{enumitem}

\usepackage{xcolor}
\usepackage[normalem]{ulem}

\newtheorem{theorem}{Theorem}[section]
\newtheorem{lemma}[theorem]{Lemma}
\newtheorem{proposition}[theorem]{Proposition}

\newtheorem*{thm*}{\protect\theoremname}
\theoremstyle{definition}
\newtheorem{definition}[theorem]{Definition}
\newtheorem{remark}[theorem]{Remark}
\newtheorem{question}[theorem]{Question}
\newtheorem*{cor*}{\protect\corollaryname}
\renewcommand{\leq}{\leqslant}
\renewcommand{\geq}{\geqslant}

\long\def\@savemarbox#1#2{\global\setbox#1\vtop{\hsize\marginparwidth 
  \@parboxrestore\tiny\raggedright #2}}
\newcommand{\PSL}{\mathsf{PSL}}
\newcommand{\PGL}{\mathsf{PGL}}
\newcommand{\GL}{\mathsf{GL}}

\newcommand{\Bc}{\mathcal B}

\newcommand{\Pc}{\mathcal P}

\newcommand{\Tc}{\mathcal T}

\DeclareMathOperator{\Hb}{\mathbb{H}}

\DeclareMathOperator{\Rb}{\mathbb{R}}

\DeclareMathOperator{\Fc}{\mathcal{F}}

\DeclareMathOperator{\SL}{\mathsf{SL}}

\DeclareMathOperator{\Hom}{Hom}
\DeclareMathOperator{\Gr}{\mathrm{Gr}}

\providecommand{\corollaryname}{Corollary}
\providecommand{\theoremname}{Theorem}

\title{Cusped Borel Anosov representations with positivity}

\author[G.-S. Lee]{Gye-Seon Lee}
\address{Department of Mathematical Sciences and Research institute of Mathematics, Seoul National University, Seoul 08826, South Korea}
\email{gyeseonlee@snu.ac.kr}

\author[T. Zhang]{Tengren Zhang}
\address{Mathematics Department, National University of Singapore}
\email{matzt@nus.edu.sg}

\begin{document}

\begin{abstract}
We show that if a cusped Borel Anosov representation from a lattice $\Gamma\subset\PGL_2(\Rb)$ to $\PGL_d(\Rb)$ contains a unipotent element with a single Jordan block in its image, then it is necessarily a (cusped) Hitchin representation. We also show that the amalgamation of a Hitchin representation with a cusped Borel Anosov representation that is not Hitchin is never cusped Borel Anosov.
\end{abstract}
\thanks{G.-S. Lee was supported by the National Research Foundation of Korea(NRF) grant funded by the Korea government(MSIT) (No. 2020R1C1C1A01013667).
T. Zhang was supported by the NUS-MOE grant A-8000458-00-00.}


\maketitle
\tableofcontents

\section{Introduction}

Let $\Gamma$ be a (word) hyperbolic group and let $\theta\subset \Delta := \{1,\dots,d-1\}$ be any subset. In his seminal work, Labourie \cite{Lab06} defined what it means for a representation $\rho:\Gamma\to \PGL_d(\Rb)$ to be $P_\theta$-Anosov. This notion of Anosov representations has proven to be very useful: It is strong enough for general theorems to be proven for the entire class of Anosov representations, but at the same time is also flexible enough to admit many interesting examples. For this reason, the theory of Anosov representations has been heavily studied and developed in the last twenty years \cite{GW12,KLP17,GGKW17,BPS19}.

Recently, there has been a successful push to generalize the notion of Anosov representations to the setting where $\Gamma$ is a relatively hyperbolic group. These include the relatively asymptotically embedded and relatively Morse representations defined by Kapovich and Leeb \cite{KL18}, the relatively dominated representations defined by Zhu \cite{Zhu19}, and the extended geometrically finite representations defined by Weisman \cite{Weis22}. Most recently, Zhu and Zimmer \cite{ZZ22a} defined the notion of relatively Anosov representations, and clarified the relationship between their notion and the other notions mentioned above.

In the case when $\Gamma\subset\PGL_2(\Rb)$ is a geometrically finite subgroup i.e. a finitely generated, non-elementary, discrete subgroup, Canary, Zhang and Zimmer \cite{CZZ21} defined the notion of a cusped Anosov representation $\rho:\Gamma\to\PGL_d(\Rb)$. If we view $\Gamma$ as a hyperbolic group relative to the cusp subgroups in $\Gamma$, then cusped Anosov representations are a special case of all the above notions. Canary, Zhang and Zimmer \cite{CZZ22} also defined the notion of transverse representations, which extends the notion of cusped Anosov representations to allow for $\Gamma$ to be any non-elementary, discrete subgroup of $\PGL_2(\Rb)$ (and more generally, any projectively visible group), see Remark \ref{rem: extend}. In this article, we will focus exclusively on transverse representations of non-elementary, discrete subgroups of $\PGL_2(\Rb)$, which we now define.

For any non-elementary, discrete subgroup $\Gamma\subset\PGL_2(\Rb)$, let $\Lambda(\Gamma)$ denote its limit set, i.e. $\Lambda(\Gamma)$ is the set of accumulation points in $\partial\Hb^2$ of some/any $\Gamma$-orbit in $\Hb^2$. Note that $\Lambda(\Gamma)$ is an infinite, $\Gamma$-invariant, compact subset of $\partial\Hb^2$. For any subset $\theta\subset\Delta$, let $\Fc_\theta(\Rb^d)$ denote the corresponding partial flag manifold, i.e. if $\theta=\{k_1,\dots, k_s\}$ with $k_1<\dots<k_s$, then
\[\Fc_\theta(\Rb^d):=\{F=(F^{k_1},\dots, F^{k_s}) \mid F^{k_i}\in\Gr_{k_i}(\Rb^d)\text{ and }F^{k_i}\subset F^{k_{i+1}}\text{ for all }i\}.\]
In the case when $\theta=\Delta$, we will simply denote $\Fc(\Rb^d):=\Fc_{\Delta}(\Rb^d)$.
 
\begin{definition}
Let $\theta\subset\Delta$ be symmetric, i.e. $k\in\theta$ if and only if $d-k\in\theta$, and let $\Gamma\subset\PGL_2(\Rb)$ be a non-elementary, discrete subgroup. A representation $\rho:\Gamma\to\PGL_d(\Rb)$ is \emph{$P_\theta$-transverse} if there is a continuous map $\xi = (\xi^k)_{k \in \theta} :\Lambda(\Gamma)\to\Fc_\theta(\Rb^d)$ that satisfies all of the following properties:
\begin{itemize}
\item $\xi$ is \emph{$\rho$-equivariant}, i.e. $\xi(\gamma \cdot x)=\rho(\gamma)\cdot \xi(x)$ for all $\gamma\in\Gamma$ and $x\in\Lambda(\Gamma)$.
\item $\xi$ is \emph{transverse}, i.e. $\xi^k(x)+\xi^{d-k}(y)=\Rb^d$ for all distinct points $x,y\in\Lambda(\Gamma)$ and all $k\in\theta$.
\item $\xi$ is \emph{strongly dynamics preserving}, i.e. if $\{\gamma_n\}$ is a sequence in $\Gamma$ such that $\gamma_n\cdot b_0\to x$ and $\gamma_n^{-1}\cdot b_0\to y$ for some/any $b_0\in\Hb^2$ and some $x,y\in\Lambda(\Gamma)$, then $\rho(\gamma_n)\cdot F\to\xi(x)$ for all $F\in\Fc_\theta(\Rb^d)$ that is transverse to $\xi(y)$. 
\end{itemize}
\end{definition}
In the above definition, the strongly dynamics preserving property of $\xi$ ensures that it is unique to $\rho$. We thus refer to $\xi$ as the \emph{limit map} of $\rho$. 

\begin{remark}\label{rem: extend}
Canary, Zhang and Zimmer \cite[Theorems 4.1 and 6.1]{CZZ21} proved that if $\Gamma\subset\PGL_2(\Rb)$ is geometrically finite, then for any symmetric $\theta\subset\Delta$, a representation $\rho:\Gamma\to\PGL_d(\Rb)$ is $P_\theta$-transverse if and only if it is cusped $P_\theta$-Anosov. 
\end{remark}

In the case when $\theta=\Delta$, $P_\Delta$-transverse representations and cusped $P_\Delta$-Anosov representations are also called \emph{Borel transverse representations} and \emph{cusped Borel Anosov representations} respectively. When $\Gamma\subset\PGL_2(\Rb)$ is a convex cocompact free subgroup, (cusped) Borel Anosov representations from $\Gamma$ to $\PGL_d(\Rb)$ can be constructed via a ping pong type argument. However, when $\Gamma\subset\PGL_2(\Rb)$ is a lattice, there are currently only two known families of cusped Borel Anosov representations: the Hitchin representations and the Barbot examples, see Section~\ref{sec: Hitchin} and Appendix \ref{app: Barbot} respectively. The search for more examples of cusped Borel Anosov representations can be formulated as the following question:
\begin{question}\label{question}
When $\Gamma\subset\PGL_2(\Rb)$ is a lattice, are there cusped Borel Anosov representations that are neither Hitchin representations nor the Barbot examples? 
\end{question}

The two main results of this paper are rigidity results about Borel transverse representations of non-elementary discrete subgroups of $\PGL_2(\Rb)$ whose limit set is all of $\partial\Hb^2$. When specialized to lattices in $\PGL_2(\Rb)$, they can be interpreted as providing supporting evidence to a negative answer to the above question. 

If $\Gamma\subset\PGL_2(\Rb)$ is a non-elementary, discrete subgroup and $\rho:\Gamma\to\PGL_d(\Rb)$ is a Hitchin representation, then it follows from the work of Canary, Zhang and Zimmer \cite{CZZ21} that $\rho$ sends every (non-identity) parabolic element in $\Gamma$ to a unipotent element with a single Jordan block, see Theorem \ref{thm: CZZ} and Remark \ref{rem: CZZ}. Our first theorem resolves Question \ref{question} under the additional assumption that the image of $\rho$ contains a unipotent element with a single Jordan block.

\begin{theorem}\label{thm 1}
Suppose that $\Gamma\subset\PGL_2(\Rb)$ is a discrete subgroup with $\Lambda(\Gamma)=\partial\Hb^2$. If $\rho:\Gamma\to\PGL_d(\Rb)$ is a Borel transverse representation whose image contains a unipotent element with a single Jordan block, then $\rho$ is a Hitchin representation.
\end{theorem}

\begin{remark}
If $\rho:\Gamma\to\PGL_d(\Rb)$ is a Barbot example, then $d$ is necessarily odd, and $\rho$ sends every parabolic element in $\Gamma$ to a unipotent element in $\PGL_d(\Rb)$ with two Jordan blocks, one of size $j$ and the other of size $d-j$ for some $j\in\{1,\dots,\frac{d-1}{2}\}$, see Appendix \ref{app: Barbot}. As such, the hypothesis of Theorem \ref{thm 1} rules out the need to consider the Barbot examples.
\end{remark}

One might attempt to construct new examples of cusped Borel Anosov representations on a lattice $\Gamma\subset\PGL_2(\Rb)$ via the following ``amalgamation" procedure. 
\begin{itemize}
\item[Step 1:] Realize $\Gamma$ as a free product of two non-elementary, geometrically finite subgroups $\Gamma_1$ and $\Gamma_2$, amalgamated over a cyclic subgroup $\langle\gamma\rangle$. 
\item[Step 2:] Specify a Barbot example $\rho_1:\Gamma_1\to\PGL_d(\Rb)$ and a Hitchin representation $\rho_2:\Gamma_2\to\PGL_d(\Rb)$ so that $\rho_1(\langle\gamma\rangle)$ is conjugate to $\rho_2(\langle\gamma\rangle)$.
\item[Step 3:] Find a cusped Borel Anosov representation $\rho:\Gamma\to\PGL_d(\Rb)$ so that $\rho|_{\Gamma_1}=\rho_1$ and $\rho|_{\Gamma_2}=\rho_2$.
\end{itemize}
There are situations (see for example \cite{GW10,CLS}) where this amalgamation procedure allows one to construct new classes of $P_\theta$-Anosov representations from existing ones. However, our next theorem implies that the amalgamation process described above will never yield a Borel transverse representation.

\begin{theorem}\label{thm 2}
Suppose that $\Gamma\subset\PGL_2(\Rb)$ is a discrete subgroup with $\Lambda(\Gamma)=\partial\Hb^2$, and let $\Gamma'\subset\Gamma$ be a non-elementary subgroup. If $\rho:\Gamma\to\PGL_d(\Rb)$ is a Borel transverse representation such that $\rho|_{\Gamma'}:\Gamma'\to\PGL_d(\Rb)$ is Hitchin, then $\rho$ is Hitchin.
\end{theorem}

By Remark \ref{rem: extend} above, Theorems \ref{thm 1} and \ref{thm 2} hold for cusped Borel Anosov representations as well; one simply imposes the additional condition that $\Gamma$ is geometrically finite.

A key tool used in the proofs of Theorem \ref{thm 1} and Theorem \ref{thm 2} (and also in the definition of Hitchin representations) is Fock and Goncharov's notion of positivity for $n$-tuples in $\Fc(\Rb^d)$ for any integer $n\geq3$, see Section \ref{sec: Fock-Goncharov positivity}. With this, one can then define the notion of a positive map from a subset $\Lambda\subset\mathbb{S}^1$ (with $\#\Lambda\geq3$) to $\Fc(\Rb^d)$: we say that a map $\xi:\Lambda\to\Fc(\Rb^d)$ is \emph{positive} if for any integer $n\geq3$, the tuple $(\xi(a_1),\dots,\xi(a_n))$ is positive for all $a_1<\dots<a_n<a_1$ in $\Lambda$ (according to the clockwise cyclic order on $\mathbb{S}^1$). The proofs of both Theorem~\ref{thm 1} and Theorem~\ref{thm 2} rely on the following result about continuous, positive maps, which is a special case of more general results of Guichard, Labourie and Wienhard \cite[Lemma 3.5 and Proposition 3.15]{GLW} in the setting of $\Theta$-positive maps. 

\begin{proposition}\label{thm 0}
Let $\xi:\mathbb{S}^1\to\Fc(\Rb^d)$ be a continuous, transverse map. If there is a pairwise distinct triple of points $x,y,z\in\mathbb{S}^1$ such that $(\xi(x),\xi(y),\xi(z))$ is positive, then $\xi$ is a positive map. 
\end{proposition}

In Section \ref{sec 1}, we will recall Fock and Goncharov's notion of positivity of $k$-tuples of complete flags and the definition of Hitchin representations. Then, in Section \ref{sec: thm 0}, we provide an elementary and self-contained proof of Proposition \ref{thm 0}. Finally, we use Proposition~\ref{thm 0} to prove Theorems~\ref{thm 1} and \ref{thm 2} in Section \ref{sec 2}. In the appendices, we give an elementary proof of a well-known fact about positive triples of flags that was used to prove Proposition \ref{thm 0}, and also describe the Barbot examples mentioned above.

\subsection*{Acknowledgements}

We are thankful for helpful conversations with Fanny Kassel and Jaejeong Lee.

\section{Positive tuples and positive maps}\label{sec 1}

\subsection{Fock-Goncharov positivity}\label{sec: Fock-Goncharov positivity}
We say that an upper triangular, unipotent matrix is \emph{totally positive} if its non-trivial minors (i.e. those that are not forced to be $0$ by virtue of the matrix being upper triangular) are positive. Then given an (ordered) basis $\Bc=(e_1,\dots,e_d)$ of $\Rb^d$, we say that a unipotent element in $\PGL_d(\Rb)$ is \emph{totally positive with respect to $\Bc$} if it is represented in the basis $\Bc$ by an upper triangular, unipotent, totally positive matrix. Let 
\[U_{>0}(\Bc)\subset\PGL_d(\Rb)\] 
denote the set of unipotent elements that are totally positive with respect to $\Bc$, and let 
\[U_{\geq0}(\Bc)\subset\PGL_d(\Rb)\] 
denote the closure of $U_{>0}(\Bc)$. Note that the elements in $U_{\geq0}(\Bc)$ are exactly the ones where all the non-trivial minors are non-negative. Using well-known formulas for how minors behave under products, it is straightforward to verify that both $U_{>0}(\Bc)$ and $U_{\geq0}(\Bc)$ are sub-semigroups of $\PGL_d(\Rb)$.

Recall that if $F,G\in\Fc(\Rb^d)$, then $F$ and $G$ are \emph{transverse} if $F^k+G^{d-k}=\Rb^d$ for all $k\in\{1,\dots,d-1\}$. When $n\geq3$, we say that an $n$-tuple of complete flags $(F_1,\dots,F_n)$ in $\Fc(\Rb^d)$ is \emph{positive} if $F_1$ and $F_n$ are transverse, and there is a basis $\Bc=(e_1,\dots,e_d)$ of $\Rb^d$ and elements $u_2,\dots,u_{n-1}\in U_{>0}(\Bc)$ such that $e_i\in F_1^i\cap F_n^{d-i+1}$ for all $i\in\{1,\dots,d\}$, and $F_j=(u_{n-1}\cdots u_j)\cdot F_n$ for all $j\in\{2,\dots,n-1\}$. The fact that $U_{>0}(\Bc)$ is a semigroup implies that if $(F_1,\dots,F_n)$ is positive, then so is $(F_1,F_{i_1},\dots,F_{i_\ell},F_n)$ for all integers $i_1,\dots,i_\ell$ such that $1<i_1<\dots<i_\ell<n$. Recall from the introduction that given a subset $\Lambda$ of $\mathbb{S}^1$, a map $\xi: \Lambda \to\Fc(\Rb^d)$ is \emph{positive} provided that if $n\geq3$ and $(x_1, \dotsc, x_n)$ is a cyclically ordered subset of pairwise distinct points in $\Lambda$, then $(\xi(x_1), \dotsc, \xi(x_n))$ is a positive $n$-tuple of flags.

The following proposition summarizes the basic properties of positive tuples of flags. It follows easily from a well-known parameterization result of Fock and Goncharov \cite[Theorem 9.1(a)]{FG} (see Kim-Tan-Zhang \cite[Observation 3.20]{KTZ}).

\begin{proposition} \label{prop: positivity basic}
Let $F_1,\dots,F_n$ be flags in $\Fc(\Rb^d)$. 
\begin{enumerate}
\item If $n\geq3$, then the following are equivalent:
\begin{itemize}
\item $(F_1,F_2,\dots,F_n)$ is positive,
\item $(F_n,\dots,F_2,F_1)$ is positive,
\item $(F_2,\dots,F_n,F_1)$ is positive,
\item $g\cdot(F_1,F_2,\dots,F_n)$ is positive for some/all $g\in\PGL_d(\Rb)$.
\end{itemize}
In particular, if $(F_1,\dots,F_n)$ is positive, then $(F_{i_1},\dots,F_{i_\ell})$ is positive for all $1\leq i_1<i_2<\dots<i_\ell\leq n$, and so $F_i$ and $F_j$ are transverse for all distinct pairs $i,j\in\{1,\dots,n\}$.
\item If $n\geq4$, then $(F_1,\dots,F_n)$ is positive if and only if $(F_1,\dots,F_{n-1})$ is positive and $(F_1,F_i,F_{n-1},F_n)$ is positive for some/all $i=2,\dots,n-2$. In particular, $(F_1,\dots,F_n)$ is positive if and only if $(F_{i_1},F_{i_2},F_{i_3},F_{i_4})$ is positive for all $1\leq i_1<i_2<i_3<i_4\leq n$.
 \end{enumerate}
\end{proposition}

Let $\Pc$ denote the set of positive triples of flags in $\Fc(\Rb^d)$, and let $\Tc$ denote the set of pairwise transverse triples of flags in $\Fc(\Rb^d)$. The following theorem is also a well-known property of positive triples of flags, which has been generalized to the setting of triples of $\Theta$-positive flags by Guichard, Labourie and Wienhard \cite[Proposition 2.5(1)]{GLW}. We provide an elementary proof in Appendix~\ref{app}. 

\begin{theorem}\label{thm: well-known}
Let $F$, $G$ and $H$ be complete flags in $\Fc(\Rb^d)$ such that both $G$ and $H$ are transverse to $F$. Let $u\in\PGL_d(\Rb)$ be the unipotent element that fixes $F$ and sends $H$ to $G$, and let $\Bc=(e_1,\dots,e_d)$ be any basis of $\Rb^d$ such that $e_k\in F^k\cap H^{d-k+1}$ for all $k\in\{1,\dots,d\}$. If $u\in U_{\geq0}(\Bc)-U_{>0}(\Bc)$, then $G$ and $H$ are not transverse. In particular, $\Pc$ is a union of connected components of $\Tc$. 
\end{theorem}

\subsection{Hitchin representations}\label{sec: Hitchin}
Suppose for now that $\Gamma\subset\PGL_2(\Rb)$ is surface group (i.e. $\Gamma$ is cocompact and torsion-free). Then the discrete and faithful representations from $\Gamma$ to $\PGL_2(\Rb)$ form a single connected component of $\Hom(\Gamma,\PGL_2(\Rb))/\PGL_2(\Rb)$, known as the Teichm\"uller component. Hitchin \cite{Hitchin} noticed that for all $d\geq2$, there is a distinguished connected component of $\Hom(\Gamma,\PGL_d(\Rb))/\PGL_d(\Rb)$ that is analogous to the Teichm\"uller component. Today, this connected component is commonly known as the \emph{Hitchin component}, and the \emph{Hitchin representations} are the ones whose conjugacy class lies in the Hitchin component. Fock and Goncharov \cite{FG} characterized the Hitchin representations as the representations for which there exists a $\rho$-equivariant positive map $\xi:\Lambda(\Gamma)\to\Fc(\Rb^d)$, and Labourie \cite{Lab06} showed that every Hitchin representation is (cusped) Borel Anosov.

Motivated by Fock and Goncharov's characterization of Hitchin representations, Canary, Zhang and Zimmer \cite{CZZ22} extended the notion of Hitchin representations to the case when $\Gamma$ is a discrete subgroup of $\PGL_2(\Rb)$. 

\begin{definition}
Let $\Gamma\subset\PGL_2(\Rb)$ be a non-elementary, discrete subgroup. A representation $\rho:\Gamma\to\PGL_d(\Rb)$ is \emph{Hitchin} if there is a continuous, $\rho$-equivariant, positive map $\xi:\Lambda(\Gamma)\to\Fc(\Rb^d)$. 
\end{definition}

Labourie's result can also be generalized to this case using the proof of \cite[Theorem~1.4]{CZZ21}.

\begin{theorem}\label{thm: CZZ}
Every Hitchin representation $\rho:\Gamma\to\PGL_d(\Rb)$ is Borel transverse, and the continuous, $\rho$-equivariant, positive map is the limit map of $\rho$ (and hence is unique). Furthermore, $\rho$ sends parabolic elements in $\Gamma$ to unipotent elements in $\PGL_d(\Rb)$ with a single Jordan block.
\end{theorem}

\begin{remark}\label{rem: CZZ}
Even though \cite[Theorem 1.4]{CZZ21} is stated only in the case when $\Gamma\subset\PGL_2(\Rb)$ is geometrically finite, the proof does not use the geometric finiteness of $\Gamma$.
\end{remark}

\section{Proof of Proposition \ref{thm 0}}\label{sec: thm 0}
To prove Proposition \ref{thm 0}, we will use the following lemma, which is already well-known to experts (see for example \cite[Proposition 3.15]{GLW}). We give an elementary proof of the lemma for the reader's convenience. We remark that the lemma is false without the continuity assumption on $\xi$.

\begin{lemma}\label{lem: positive map}
If $\xi:\mathbb{S}^1\to\Fc(\Rb^d)$ is a continuous map such that $(\xi(a),\xi(b),\xi(c))$ is positive for every pairwise distinct triple $a,b,c\in\mathbb{S}^1$, then $\xi$ is a positive map.
\end{lemma}

\begin{proof}
By Proposition \ref{prop: positivity basic}(2), it suffices to show that $(\xi(x),\xi(y),\xi(z),\xi(w))$ is positive for all quadruples $x,y,z,w\in\mathbb{S}^1$ such that $x<y<z<w<x$ along $\mathbb{S}^1$. Pick any such quadruple $x,y,z,w\in\mathbb{S}^1$, and let $I\subset\mathbb{S}^1$ denote the closed subinterval that contains $z$ with endpoints $y$ and $w$. By Proposition \ref{prop: positivity basic}(1), the map $\xi$ is transverse. Thus, for all $t\in I$, we may define the map
\[u:I\to\PGL_d(\Rb)\]
by setting $u(t)\in\PGL_d(\Rb)$ to be the unipotent element that fixes $\xi(x)$ and sends $\xi(w)$ to $\xi(t)$. The continuity of $\xi$ then implies that the map $u$ is continuous. 

Since $(\xi(x),\xi(y),\xi(w))$ is positive, there is a basis $\Bc=(e_1,\dots,e_d)$ such that $e_k\in\xi(x)^k\cap\xi(w)^{d-k+1}$ and $u(y)\in U_{>0}(\Bc)$. First, we prove that $u(z)\in U_{>0}(\Bc)$ as well. If this were not the case, then the continuity of $u$ implies that there is some $t_0\in (y,z]\subset I$ such that $u(t_0)\in U_{\geq0}(\Bc)-U_{>0}(\Bc)$. By Theorem \ref{thm: well-known}, $\xi(t_0)$ and $\xi(w)$ are not transverse, thus contradicting the fact that $\xi$ is a transverse map.

Next, we show that $u(z)^{-1}u(y)\in U_{>0}(\Bc)$ as well. To do so, let 
\[v:[z,w]\to\PGL_d(\Rb)\] 
be the continuous map defined by $v(t):=u(t)^{-1}u(y)$. Observe that $v(w)=u(y)\in U_{>0}(\Bc)$. Thus, if $u(z)^{-1}u(y)=v(z)\notin U_{>0}(\Bc)$, then there is some $t_0\in[z,w)$ such that $v(t_0)\in U_{\geq0}(\Bc)-U_{>0}(\Bc)$. By Theorem \ref{thm: well-known}, the pair of flags $\xi(w)$ and $v(t_0)\cdot\xi(w)$ are not transverse, which means that $\xi(t_0)=u(t_0)\cdot\xi(w)$ and $\xi(y)=u(t_0)v(t_0)\cdot\xi(w)$ are not transverse. This contradicts the fact that $\xi$ is a transverse map.

Since we have proven that both $u(z)$ and $u(z)^{-1}u(y)$ lie in $U_{>0}(\Bc)$, the quadruple of flags
\[\big(\xi(x),\xi(y),\xi(z),\xi(w)\big)=\big(\xi(x),u(z)u(z)^{-1}u(y)\cdot\xi(w),u(z)\cdot\xi(w),\xi(w)\big)\] 
is positive, so the lemma follows.
\end{proof}

\begin{proof}[Proof of Proposition \ref{thm 0}]
By Lemma \ref{lem: positive map}, it suffices to show that $(\xi(a),\xi(b),\xi(c))$ is positive for any pairwise distinct triple $a,b,c\in\mathbb{S}^1$. By Proposition \ref{prop: positivity basic}(1), we may assume that $a<b<c$ and $x<y<z$ by switching the roles of $a$ and $c$ and the roles of $x$ and $z$ if necessary. Then there are continuous maps 
\[f_1,f_2,f_3:[0,1]\to\mathbb{S}^1\] 
such that $(f_1(0),f_2(0),f_3(0))=(x,y,z)$, $(f_1(1),f_2(1),f_3(1))=(a,b,c)$, and $(f_1(t), f_2(t),f_3(t))$ are pairwise distinct triples for all $t$. 

Recall that $\Pc$ denotes the set of positive triples of flags in $\Fc(\Rb^d)$, and $\Tc$ denotes the set of pairwise transverse triples of flags in $\Fc(\Rb^d)$. Since $\xi$ is continuous and transverse, this implies that the map
\[F:[0,1]\to\Tc\]
given by $F(t)=\big(\xi(f_1(t)),\xi(f_2(t)),\xi(f_3(t))\big)$ is well-defined and continuous. Since $F(0)\in\Pc$ by hypothesis, Theorem \ref{thm: well-known} implies that $F(1)\in\Pc$. 
\end{proof}

\section{Proof of Theorems \ref{thm 1} and \ref{thm 2}}\label{sec 2}
Using Proposition \ref{thm 0}, we will now prove Theorems \ref{thm 1} and \ref{thm 2}. 

\begin{proof}[Proof of Theorem \ref{thm 1}] 
The \emph{$d$-th upper triangular Pascal matrix} $Q_d$ is the $d\times d$ upper triangular matrix whose $(i,j)$-th entry (with $i\leq j$) is the integer $\binom{j-1}{i-1}$. To prove the theorem, we will first recall some basic properties of $Q_d$.

\begin{lemma}\label{lem: Pascal}
$Q_d$ is totally positive, unipotent, and has a single Jordan block
\end{lemma}

\begin{proof}
The claim that $Q_d$ is unipotent is obvious, and the claim that $Q_d$ has a single Jordan block is a straightforward calculation: one simply verifies that $Q_d$ has a unique eigenvector. 

To prove that $Q_d$ is totally positive, observe that the natural $\GL_2(\Rb)$ action on the symmetric tensor ${\rm Sym}^{d-1}(\Rb^2)$ given by 
\[g(v_1\odot\dots\odot v_{d-1}):=g(v_1)\odot\dots\odot g(v_{d-1})\] 
induces a representation 
\[\iota_d:\GL_2(\Rb)\to\GL(\mathrm{Sym}^{d-1}(\Rb^2))\cong\GL_d(\Rb).\] 
Here, the identification $\GL(\mathrm{Sym}^{d-1}(\Rb^2))\cong\GL_d(\Rb)$ is induced by the linear identification 
\[{\rm Sym}^{d-1}(\Rb^2)\cong\Rb^d\]
given by identifying the standard basis $(e_1,\dots,e_d)$ of $\Rb^d$ with the basis $(e_1^{d-1},e_1^{d-2}e_2,\dots,e_1e_2^{d-2},e_2^{d-1})$ of ${\rm Sym}^{d-1}(\Rb^2)$ induced by the standard basis $(e_1,e_2)$ of $\Rb^2$. Note that the representation $\iota_d$ descends to a representation, also denoted
\[\iota_d:\PGL_2(\Rb)\to\PGL_d(\Rb).\]

If we take $\Bc$ to be the standard basis of $\Rb^d$, then by \cite[Proposition 5.7]{FG}, 
\[\iota_d(U_{>0}(e_1,e_2))\subset U_{>0}(\Bc).\]
It is also straightforward to verify that $[Q_d]=\iota_d\left(\begin{bmatrix}
1&1\\
0&1
\end{bmatrix}\right)$
and that
$\begin{bmatrix}
1&1\\
0&1
\end{bmatrix}$ clearly lies in $U_{>0}(e_1,e_2)$. Thus, $[Q_d]\in U_{>0}(\Bc)$, so $Q_d$ is totally positive.
\end{proof}

The proof of this theorem relies on the following lemma, which demonstrates the inherent positive nature of a unipotent element in $\PGL_d(\Rb)$ with a single Jordan block.

\begin{lemma}\label{lem: unipotent}
Let $u\in\PGL_d(\Rb)$ be a unipotent element with a single Jordan block, and let $F$ be the fixed flag of $u$. Then for any flag $G$ that is transverse to $F$ and for any sufficiently large $t$, the triple $(F,u^t\cdot G,G)$ is positive.
\end{lemma}

\begin{proof}
By Lemma \ref{lem: Pascal}, $Q_d$ is a unipotent upper triangular matrix with a single Jordan block, so we may choose a basis $\Bc=(f_1,\dots,f_d)$ of $\Rb^d$ such that $u$ is represented in $\Bc$ by $Q_d$. Then $u^t$ is represented in $\Bc$ by the matrix $Q_d^t$, which is upper triangular, and whose $(i,j)$-th entry (with $i\leq j$) is $\binom{j-1}{i-1}t^{j-i}$. Furthermore, for all $k\in\{1,\dots,d-1\}$, the subspace $F^k\subset\Rb^d$ is spanned by $\{f_1,\dots,f_k\}$. 

Let $H\in\Fc(\Rb^d)$ be the flag such that for all $k\in\{1,\dots,d-1\}$, the subspace $H^k\subset\Rb^d$ is spanned by $\{f_{d-k+1},\dots,f_d\}$. Since $G$ is transverse to $F$, there is some unipotent $v\in\PGL_d(\Rb)$ that fixes $F$ and sends $H$ to $G$. It is now sufficient to verify that $v^{-1}u^tv\in U_{>0}(\Bc)$ for sufficiently large $t$. Indeed, if this were the case, then the observation that
\[v^{-1}\cdot(F,u^t\cdot G,G)=(F,v^{-1}u^tv\cdot H,H)\]
implies that $(F,u^t\cdot G,G)$ is positive for sufficiently large $t$.

In fact, we will show that if $v'$ and $v$ are two unipotent elements that fix $F$, then $v'u^tv\in U_{>0}(\Bc)$ for sufficiently large $t$. Observe that since $v'$ and $v$ are represented in the basis $\Bc$ by upper triangular matrices whose diagonal entries are all $1$, the product $v'u^tv$ is also represented in the basis $\Bc$ by an upper triangular matrix $M_t$ whose diagonal entries are all $1$. Furthermore, for each $i<j$, the $(i,j)$-th entry of $M_t$ is a polynomial in the variable $t$ whose leading term is $\binom{j-1}{i-1}t^{j-i}$, which is the $(i,j)$-th entry of $Q_d^t$. By Lemma \ref{lem: Pascal}, $Q_d^t$ is totally positive, so the leading term of any minor of $M_t$ is the corresponding minor of $Q_d^t$. Hence, for sufficiently large $t$, we have $v'u^tv\in U_{>0}(\Bc)$. 
\end{proof}

Let $\gamma\in\Gamma$ be the element such that $\rho(\gamma)$ is unipotent with a single Jordan block. Since $\rho$ is Borel transverse, the strongly dynamics preserving property of its limit map $\xi:\Lambda(\Gamma)\to\Fc(\Rb^d)$ ensures that $\gamma$ is parabolic. Let $x\in\Lambda(\Gamma)$ be the unique fixed point of $\gamma$ and let $y\in\Lambda(\Gamma)-\{x\}$. Then $\xi(x)$ is the fixed flag of $\rho(\gamma)$. By Lemma \ref{lem: unipotent} and the $\rho$-equivariance and transversality of $\xi$, the triple of flags $\big(\xi(x),\xi(\gamma^ny),\xi(y)\big)$ is positive for sufficiently large $n$. Proposition \ref{thm 0} then implies that $\xi$ is a positive map, so $\rho$ is a Hitchin representation.
\end{proof}

\begin{proof}[Proof of Theorem \ref{thm 2}]

Let $\Lambda(\Gamma')\subset\Lambda(\Gamma)$ be the limit set of $\Gamma'$, and let $x,y,z$ be pairwise distinct points in $\Lambda(\Gamma')$ (this exists because $\Gamma'$ is non-elementary). Since $\rho$ is Borel transverse with limit map $\xi:\Lambda(\Gamma)\to\Fc(\Rb^d)$, note that $\rho|_{\Gamma'}$ is also Borel transverse with limit map $\xi|_{\Lambda(\Gamma')}:\Lambda(\Gamma')\to\Fc(\Rb^d)$. Since $\rho|_{\Gamma'}$ is Hitchin, the map $\xi|_{\Lambda(\Gamma')}$ is a positive map. Therefore, $\big(\xi(x),\xi(y),\xi(z)\big)$ is a positive triple, so Proposition \ref{thm 0} implies that $\xi$ is a positive map. As such, $\rho$ is a Hitchin representation.
\end{proof}

\appendix

\section{Proof of Theorem \ref{thm: well-known}}\label{app}
In this proof, we fix the basis $\Bc$, and hence may view every $u\in U_{\geq0}(\Bc)$ as a unipotent upper triangular $d \times d$ matrix. Given (strictly) increasing tuples 
\[I = (i_1, \dotsc, i_k)\quad\text{and}\quad J = (j_1, \dotsc, j_\ell)\] 
of integers (weakly) between $1$ and $d$, we denote by $u_{I,J}$ the submatrix of $u$ corresponding to the $I$ rows and $J$ columns. We say that $I$ is \emph{consecutive} if $i_p = i_1 + p -1$ for each $p \in \{1, \dotsc, k\}$. If $k>1$, we also denote $I':=(i_1,\dots,i_{k-1})$ and $I'':=(i_2,\dots,i_k)$.

\begin{lemma}\label{lem:consecutive}
Let $u\in U_{\geq0}(\Bc)$ and let $k\in\{1,\dots,d\}$. Suppose that all the non-trivial $\ell\times\ell$-minors of $u$ are positive for all $\ell<k$. If all the non-trivial $k\times k$ minors $\det(u_{I,J})$ of $u$ with consecutive $I$ and consecutive $J$ are positive, then all the non-trivial $k\times k$ minors of $u$ are positive.
\end{lemma}

\begin{proof}
Notice that it suffices to prove the following pair of claims (assuming that all the non-trivial $\ell\times\ell$-minors of $u$ are positive for all $\ell<k$):
\begin{enumerate}
\item Fix $I$ of length $k$. If all the non-trivial $k\times k$ minors of $u$ of the form $\det(u_{I,J})$ with consecutive $J$ are positive, then all the non-trivial $k\times k$ minors of $u$ of the form $\det(u_{I,J})$ are positive.
\item Fix $J$ of length $k$. If all the non-trivial $k\times k$ minors of $u$ of the form $\det(u_{I,J})$ with consecutive $I$ are positive, then all the non-trivial $k\times k$ minors of $u$ of the form $\det(u_{I,J})$ are positive.
\end{enumerate}
Indeed, if all the non-trivial $k\times k$ minors $\det(u_{I,J})$ of $u$ with consecutive $I$ and consecutive $J$ are positive, then we may apply Claim (1) to deduce that all the non-trivial $k\times k$ minors $\det(u_{I,J})$ of $u$ with consecutive $I$ are positive. Applying Claim (2) now gives the desired conclusion. 

We only prove Claim (1); the proof of Claim (2) is the same, except that the roles of $I$ and $J$ are switched.

When $k=1$, Claim (1) is obvious because every tuple of length $1$ is consecutive. We may thus assume that $k\in\{2,\dots,d\}$. Denote $J=(j_1,\dots,j_k)$, and notice that 
\[m:=j_k-j_1+1\in\{k,\dots,d\}.\] 
We will proceed by induction on $m$.

In the base case when $m=k$, $J$ is consecutive, so $\det(u_{I,J})$ is positive by assumption. 

For the inductive step, fix $m\in\{k+1,\dots,d\}$. Since $k < m$, $J$ is not consecutive, so there exist $q \in \{ 1, \dotsc, k-1\}$ and an integer $n$ such that $j_{q} < n < j_{q+1}$. Suppose for the purpose of contradiction that $\det(u_{I,J})=0$. Then we may write
\begin{align}\label{eqn: appendix A0} c_1 u_{I,j_1} + \cdots +  c_k u_{I,j_k} = \vec{0} \end{align}
for some $c_1,\dots,c_k \in \mathbb{R}$ that are not all zero. Thus, 
\begin{eqnarray}\label{eqn: appendix A1}
0 & = & \det( \vec{0}, u_{I,j_2}, \dotsc, u_{I,j_q}, u_{I,n}, u_{I,j_{q+1}}, \dotsc, u_{I,j_{k-1}} )\nonumber \\ 
  & = & \det( c_1 u_{I,j_1} + \cdots +  c_k u_{I,j_k}, u_{I,j_2}, \dotsc, u_{I,j_q}, u_{I,n}, u_{I,j_{q+1}}, \dotsc, u_{I,j_{k-1}} )\\
  & = & c_1 \det( u_{I,(j_1,j_2,\dots,j_q,n,j_{q+1},\dots,j_{k-1})}) + (-1)^{k-1} c_k \det(u_{I,(j_2,\dots,j_q,n,j_{q+1},\dots,j_{k-1},j_k)}). \nonumber 
\end{eqnarray}

Since $\det(u_{I,J})$ is a non-trivial $k\times k$ minor of $u$, i.e. $i_p \leq j_p$ for all $p \in \{1, \dotsc, k\}$, both $\det(u_{I',J'})$ and $\det(u_{I',J''})$ are non-trivial $(k-1)\times (k-1)$ minors of $u$, so they are both positive by assumption. So, \eqref{eqn: appendix A0} implies that $c_1\neq 0\neq c_k$. At the same time, notice that $j_k-j_2+1<m$. Since $\det(u_{I,J})$ is a non-trivial $k\times k$-minor of $u$, the same is true for $\det(u_{I,(j_2,\dots,j_q,n,j_{q+1},\dots,j_k)})$, so it is positive by the inductive hypothesis. It now follows from \eqref{eqn: appendix A1} that 
\[(-1)^k \frac{c_k}{c_1}=\frac{\det( u_{I,(j_1,\dots,j_q,n,j_{q+1},\dots,j_{k-1})})}{\det(u_{I,(j_2,\dots,j_q,n,j_{q+1},\dots,j_k)})} \geq 0.\] 

On the other hand, we also have
\begin{eqnarray}
0 & = & \det( \vec{0}, u_{I',j_2}, \dotsc, u_{I',j_{k-1}} )\nonumber \\ 
  & = & \det( c_1 u_{I',j_1} + \cdots +  c_k u_{I',j_k}, u_{I',j_2}, \dotsc, u_{I',j_{k-1}} ) \nonumber\\
  & = & c_1 \det( u_{I',J'} ) + (-1)^{k-2} c_k \det(u_{I',J''}), \nonumber 
\end{eqnarray}
so
\[(-1)^{k-1} \frac{c_k}{c_1}=\frac{\det(u_{I',J'})}{\det(u_{I',J''})} > 0\]
because $c_1\neq 0\neq c_k$ and both $\det(u_{I',J'})$ and $\det(u_{I',J''})$ are positive.  
We thus arrive at a contradiction, so $\det(u_{I,J})\neq 0$. Since $u\in U_{\geq0}(\Bc)$, it follows that $\det(u_{I,J})>0$. This completes the inductive step.
\end{proof}

\begin{lemma}\label{lem:righttopcorner}
Let $u\in U_{\geq0}(\Bc)$. Suppose that there exists $k \in \{1, \dotsc, d\}$ such that 
\begin{itemize}
\item all the non-trivial $\ell\times\ell$-minors of $u$ are positive for all $\ell<k$;
\item there is a non-trivial $k \times k$ minor $\det(u_{I,J})$ of $u$ such that $I$ and $J$ are consecutive and $\det(u_{I,J}) = 0$. 
\end{itemize}
Then $\det(u_{I_0,J_0})=0$, where $I_0 = (1, \dotsc, k)$ and $J_0 = (d-k+1, \dotsc, d)$.
\end{lemma}

\begin{proof}
Notice that it suffices to prove that $\det(u_{I,J_0})=0$ and that $\det(u_{I_0,J})=0$. We will only prove the former; the proof of the latter is the same.

Let $I = (i_1, \dotsc, i_k)$ and $J = (j_1, \dotsc, j_k)$. By assumption, $\det\left(u_{I',J'}\right) >0$, so $u_{I',j_1},\dots,u_{I',j_{k-1}}$ and hence $u_{I,j_1},\dots,u_{I,j_{k-1}}$ is a linearly independent collection of vectors. Since $\det\left(u_{I,J}\right) =0$, it follows that $u_{I,j_k}$ is a linear combination of $u_{I,j_1},\dots,u_{I,j_{k-1}}$.

Fix $n \in \{ j_k+1, \dotsc, d \}$. Since $\det(u_{I,J})$ is a non-trivial minor and $\det(u_{I,J})=0$, we have $i_k<j_k$. Then
\begin{eqnarray*}
0&\leq& \det\left(u_{(i_1,\dots,i_k,j_k),(j_1,\dots,j_k,n)}\right)\\
&=&\det \left(\begin{smallmatrix} u_{I,J'} & u_{I,j_{k}} & u_{I,n} \\  \vec{0} & 1 & u_{j_k,n} \end{smallmatrix} \right)\\
&=&u_{j_k,n}\det(u_{I,J})-\det\left(u_{I,(j_1,\dots,j_{k-1},n)}\right)\\
&=&-\det\left(u_{I,(j_1,\dots,j_{k-1},n)}\right)
\end{eqnarray*}
where the first inequality holds because $u\in U_{\geq0}(\Bc)$. At the same time, $\det\left(u_{I,(j_1,\dots,j_{k-1},n)}\right) \geq0$ because $u\in U_{\geq0}(\Bc)$, so $\det\left(u_{I,(j_1,\dots,j_{k-1},n)}\right) = 0$.  It follows that $u_{I,n}$ is a linear combination of the linearly independent collection of vectors $u_{I,j_1}, \dotsc, u_{I,j_{k-1}}$.

Since $J$ is consecutive, we have proven that the $k$ vectors $u_{I,d-k+1}, u_{I,d-k+2},\dots,u_{I,d}$ are all linear combinations of $u_{I,j_1}, \dotsc, u_{I,j_{k-1}}$. In particular, their span has dimension $k-1$, so $\det(u_{I,J_0})=0$. 
\end{proof}

\begin{proof}[Proof of Theorem \ref{thm: well-known}]
Since $u\in U_{\geq0}(\Bc)-U_{>0}(\Bc)$, there is some $k\in\{1,\dots,d-1\}$ such that some non-trivial $k\times k$ minor $\det(u_{I,J})$ of $u$ is zero, while all the non-trivial $\ell\times\ell$-minors of $u$ are positive for all $\ell<k$.

By Lemma \ref{lem:consecutive}, we may assume that both $I$ and $J$ are consecutive. Then Lemma \ref{lem:righttopcorner} implies $\det(u_{I_0,J_0}) = 0$ with $I_0 = (1, \dotsc, k)$ and $J_0 = (d-k+1, \dotsc, d)$. Therefore, the span of the vectors $u_{I_0,d-k+1},\dots,u_{I_0,d}$ has dimension at most $k-1$, so
\begin{eqnarray}
G^k + H^{d-k} &=& u \cdot {\rm Span}(e_d, \dotsc, e_{d-k+1}) + {\rm Span}(e_d, \dotsc, e_{k+1})\nonumber \\ 
                    &=&  {\rm Span}\left(u\cdot e_d, \dotsc, u\cdot e_{d-k+1},e_d, \dotsc, e_{k+1}\right) \nonumber \\ 
                    &=&  {\rm Span}\left( \left(\begin{smallmatrix} u_{I_0,d} \\ \vec{0} \end{smallmatrix}\right), \dotsc, \left(\begin{smallmatrix} u_{I_0,d-k+1} \\ \vec{0} \end{smallmatrix}\right),e_d, \dotsc, e_{k+1}\right)\nonumber  \\
                    &\neq& \mathbb{R}^d. \nonumber 
\end{eqnarray}
This implies that $G$ and $H$ are not transverse.
\end{proof}

\section{The Barbot examples}\label{app: Barbot}

Fix a lattice $\Gamma\subset\SL_2(\Rb)$ and some odd integer $d>2$. In this appendix, we define the Barbot examples, which are representations $\rho:\Gamma\to\PGL_d(\Rb)$ that are Borel transverse (or equivalently, cusped Borel Anosov), but not Hitchin. These are a straightforward generalization of examples (due to Barbot \cite{Barbot}) of Borel Anosov representations of a surface group into $\PGL_3(\Rb)$ that are not Hitchin.

 To define the Barbot examples, we need some preliminary results. First, let $(e_1,\dots,e_d)$ be the standard basis of $\Rb^d$, and equip $\Rb^d$ with the standard inner product. For any $g\in\PGL_d(\Rb)$, let 
\[\sigma_1(g)\geq\dots\geq\sigma_d(g) > 0 \]
denote the singular values of (any unit-determinant, linear representative of) $g$, and let 
\[A_g:={\rm diag}(\log\sigma_1(g),\dots,\log\sigma_d(g)).\]
By the singular value decomposition theorem, we may write every $g\in\PGL_d(\Rb)$ as the product
\[g=m\exp(A_g)\ell\]
for some $m,\ell\in\mathsf{PO}(d)$ (which are not necessarily unique). For every $g\in\PGL_d(\Rb)$, choose $m_g,\ell_g\in\mathsf{PO}(d)$ such that $g=m_g\exp(A_g)\ell_g$.

Let $F_0\in\Fc(\Rb^d)$ be the flag such that 
\[F_0^k={\rm Span}(e_1,\dots,e_k)\]
for all $k\in\{1,\dots,d-1\}$, and define
\[U(g):=m_g\cdot F_0.\]
One can verify that if $\sigma_k(g)>\sigma_{k+1}(g)$ for all $k\in\{1,\dots,d-1\}$, then $U(g)$ does not depend on the choice of $m_g$ and $\ell_g$, and hence is canonical to $g$. The following proposition is a standard linear algebra fact, see \cite[Appendix A]{CZZ23} for a proof.

\begin{proposition}\label{prop: singular values}
Let $\{g_n\}$ be a sequence in $\PGL_d(\Rb)$ and $F_+,F_-\in\Fc(\Rb^d)$. The following are equivalent:
\begin{enumerate}
\item $U(g_n)\to F_+$, $U(g_n^{-1})\to F_-$, and $\frac{\sigma_k(g_n)}{\sigma_{k+1}(g_n)}\to\infty$ for all $k\in\{1,\dots,d-1\}$.
\item $g_n(F)\to F_+$ for all $F$ transverse to $F_-$, and $g_n^{-1}(F)\to F_-$ for all $F$ transverse to $F_+$. 
\end{enumerate}
\end{proposition}

Next, recall that $g\in\PGL_d(\Rb)$ is \emph{weakly unipotent} if its multiplicative Jordan-Chevalley decomposition has elliptic semisimple part and non-trivial unipotent part. We say that a representation $\rho:\Gamma\to\PGL_d(\Rb)$ is \emph{type preserving} if it sends parabolic elements in $\Gamma$ to weakly unipotent elements in $\PGL_d(\Rb)$. If $\Gamma\subset\SL_2(\Rb)$ is geometrically finite, then given a type preserving representation $\sigma:\Gamma\to\PGL_d(\Rb)$, one can define
\[\Hom_{\rm tp}(\sigma)\subset \Hom(\Gamma,\PGL_d(\Rb))\]
to be the set of representations $\rho:\Gamma\to\PGL_d(\Rb)$ such that $\rho(\alpha)$ is conjugate to $\sigma(\alpha)$ for all parabolic $\alpha\in\Gamma$. The following are results of Canary, Zhang and Zimmer \cite[Theorem 4.1(2) and Theorem 8.1]{CZZ21}

\begin{theorem}[Canary-Zhang-Zimmer]\label{thm: open}
Suppose that $\Gamma\subset\SL_2(\Rb)$ is geometrically finite. If $\rho:\Gamma\to\PGL_d(\Rb)$ is $P_\theta$-transverse for some symmetric $\theta\subset\Delta$, then:
\begin{enumerate}
\item $\rho$ is type-preserving.
\item The set of $P_\theta$-transverse representations in $\Hom_{\rm tp}(\rho)$ is open.
\end{enumerate}
\end{theorem}

Finally, let $k\geq 1$ be an integer. Recall from the proof of Lemma \ref{lem: Pascal} the representation 
\[\iota_k:\GL_2(\Rb)\to\GL(\mathrm{Sym}^{k-1}(\Rb^2))\cong\GL_k(\Rb).\] 
One can verify that $\iota_k$ restricts to a representation 
\[\iota_k:\SL_2(\Rb)\to\SL_k(\Rb).\]
Now, given any $j\in\{1,\dots,\frac{d-1}{2}\}$, let 
\[\tau_{d,j}:=\iota_{d-j}\oplus\iota_j:\SL_2(\Rb)\to\SL_{d-j}(\Rb)\oplus\SL_{j}(\Rb)\subset\SL_d(\Rb).\]
Let $(e_1,\dots,e_d)$ be the standard basis of $\Rb^d$, let 
\[(f_1,\dots,f_{d-j}):=(e_1,e_2,\dots,e_{d-j})\quad\text{and}\quad(f_1',\dots,f_j'):=(e_{d-j+1},\dots,e_d),\] 
and let {$k:=\frac{d-2j+1}{2}$}. Then let $B'\subset\SL_d(\Rb)$ be the upper triangular group with respect to the basis
\[\Bc:=(f_1,f_2,\dots,f_k,f_1',f_{k+1},f_2',f_{k+2},\dots,f_j', f_{k+j}, f_{k+j+1}, f_{k+j+2}, \dots, f_{d-j})\]
of $\Rb^d$. Observe that $\tau_{d,j}^{-1}(B')$ is the upper triangular subgroup of $\SL_2(\Rb)$ with respect to the standard basis $(e_1,e_2)$ of $\Rb^2$, so we may define the $\tau_{d,j}$-equivariant embedding
\[\xi_{d,j}:\mathbb{RP}^1\cong\SL_2(\Rb)/\tau_{d,j}^{-1}(B')\to\SL_d(\Rb)/B'\cong\Fc(\Rb^d).\] 
Let $F_+$ and $F_-$ be the flags in $\Fc(\Rb^d)$ with the defining property that for all $k\in\{1,\dots,d-1\}$, $F_+^k$ is spanned by the first $k$ vectors of the basis $\Bc$ and $F_-^k$ is spanned by the last $k$ vectors of $\Bc$. Observe that $\xi_{d,j}([e_1])=F_+$ and $\xi_{d,j}([e_2])=F_-$.

\begin{proposition}\label{prop: Barbot} 
For every $j\in\{1,\dots,\frac{d-1}{2}\}$, the following hold:
\begin{enumerate}
\item The map $\xi_{d,j}$ is transverse.
\item If $\{g_n\}$ is a sequence in $\SL_2(\Rb)$ and $x,y\in\mathbb{RP}^1$ such that $g_n\cdot b_0\to x$ and $g_n^{-1}\cdot b_0\to y$ for some/all $b_0\in\Hb^2$, then $\tau_{d,j}(g_n)\cdot F\to\xi_{d,j}(x)$ for all $F$ transverse to $\xi_{d,j}(y)$, and $\tau_{d,j}(g_n^{-1})\cdot F\to\xi_{d,j}(y)$ for all $F$ transverse to $\xi_{d,j}(x)$.
\end{enumerate}
In particular, if $\Gamma\subset\SL_2(\Rb)$ is a non-elementary, discrete subgroup and $\pi:\SL_d(\Rb)\to\PSL_d(\Rb)\subset\PGL_d(\Rb)$ is the obvious quotient map, then 
\[\rho:=\pi\circ\tau_{d,j}|_{\Gamma}:\Gamma\to\PGL_d(\Rb)\] 
is Borel transverse with limit map $\xi_{d,j}|_{\Lambda(\Gamma)}$.
\end{proposition}

\begin{proof}
To simplify notation, we will denote $\xi:=\xi_{d,j}$ and $\tau:=\tau_{d,j}$. 

(1) Pick any pair of distinct points $a,b\in\mathbb{RP}^1$. Then there is some $g\in\SL_2(\Rb)$ such that $(g\cdot a,g\cdot b)=([e_1],[e_2])$. By the $\tau$-equivariance of $\xi$, it follows that
\[(\xi(a),\xi(b))=\left(\tau(g^{-1})\cdot \xi([e_1]),\tau(g^{-1})\cdot \xi([e_2])\right),\]
so it suffices to verify that $\xi([e_1])$ and $\xi([e_2])$ are transverse. This holds because $\xi([e_1])=F_+$ and $\xi([e_2])=F_-$.

(2) Note that $g_n\cdot z\to x$ for all $z\in\mathbb{RP}^1-\{y\}$ and  $g_n^{-1}\cdot z\to y$ for all $z\in\mathbb{RP}^1-\{x\}$. Proposition~\ref{prop: singular values} then implies that 
\[m_n\cdot [e_1]=U(g_n)\to x,\quad\ell_n^{-1}\cdot[e_2]=U(g_n^{-1})\to y \quad \text{ and } \quad \frac{\sigma_1(g_n)}{\sigma_2(g_n)}\to\infty,\]
where $g_n=m_n\exp(A_{g_n})\ell_n$ is a singular value decomposition of $g_n$.  
In particular, any subsequential limit $m$ of $\{m_n\}$ and $\ell$ of $\{\ell_n\}$ satisfy
\[m\cdot[e_1]=x \quad \textrm{and} \quad  \ell^{-1}\cdot[e_2]=y.\]

Note that 
\[\tau(g_n)=\tau(m_n)\tau(\exp(A_{g_n}))\tau(\ell_n)\] 
is a singular value decomposition of $\tau(g_n)$. It then follows that
\[U(\tau(g_n))=\tau(m_n)\cdot F_+\to \tau(m)\cdot F_+=\tau(m)\cdot \xi([e_1])=\xi(x),\]
where $m$ is some/any subsequential limit of $\{m_n\}$. Similarly,
\[U(\tau(g_n)^{-1})\to\xi(y).\]
This also implies that 
$$\tau(\exp(A_{g_n}))=\exp(A_{\tau(g_n)}) \quad \textrm{and} \quad  
\frac{\sigma_i(\tau(g_n))}{\sigma_{i+1}(\tau(g_n))} \to \infty,$$
because
\[\frac{
\sigma_i(\tau(g_n))}{\sigma_{i+1}(\tau(g_n))}=\left\{
\begin{array}{rl}
\frac{\sigma_1(g_n)}{\sigma_2(g_n)}&\text{ if } 1\leq i\leq k-1\text{ or } d-k+1\leq i\leq d-1,\vspace{5pt}\\
\sqrt{\frac{\sigma_1(g_n)}{\sigma_2(g_n)}}&\text{ if } k\leq i\leq d-k.
\end{array}\right.\]

Thus, by Proposition \ref{prop: singular values}, $\tau(g_n)\cdot F\to\xi(x)$ for all $F$ transverse to $\xi(y)$ and $\tau(g_n)\cdot F\to\xi(y)$ for all $F$ transverse to $\xi(x)$. 

Therefore, $\rho$ is Borel transverse with limit map $\xi|_{\Lambda(\Gamma)}$. Indeed, $\xi$ is continuous and $\tau$-equivariant, $\xi$ is transverse by (1), and $\xi$ is strongly dynamics preserving by (2).
\end{proof}

We may now define the Barbot examples. Given $j\in\{1,\dots,\frac{d-1}{2}\}$, a representation $\rho:\Gamma\to\PGL_d(\Rb)$ is a \emph{$(\Gamma,d,j)$-Barbot example} if there is a continuous path $f:[0,1]\to\Hom_{\rm tp}(\pi\circ\tau_{d,j}|_\Gamma)$ such that $f(0)=\rho$, $f(1)=\pi\circ\tau_{d,j}|_{\Gamma}$, and $f(t)$ is Borel transverse for all $t\in[0,1]$. By Theorem~\ref{thm: open} and Proposition \ref{prop: Barbot}, the $(\Gamma, d,j)$-Barbot examples form a connected, non-empty, open set in $\Hom_{\rm tp}(\pi\circ\tau_{d,j}|_\Gamma)$.

\begin{remark}
We may define the $(\Gamma,d,j)$-Barbot examples for discrete subgroups $\Gamma\subset\SL^\pm_2(\Rb)$ as well: these are representations $\rho:\Gamma\to\PGL_d(\Rb)$ whose restriction to $\Gamma\cap\SL_2(\Rb)$ is a $(\Gamma,d,j)$-Barbot example as described above. Since $\Gamma\cap\SL_2(\Rb)\subset\Gamma$ is a finite-index subgroup, these representations are also Borel-transverse.
\end{remark}

\bibliographystyle{alpha}
\bibliography{bibliography}

\end{document}